\documentclass[12pt]{article}

\usepackage{amsmath,amssymb,amsthm,amsrefs}

\title{Reverse mathematics and Weihrauch analysis motivated by finite complexity theory}
\author{Zack BeMent,
Jeffry Hirst\thanks{Corresponding author:  Jeffry Hirst hirstjl@appstate.edu  \newline Authors' pre-peer-review draft.},
and Asuka Wallace}

\date{May 4, 2021}

\theoremstyle{plain}
\newtheorem{theorem}{Theorem}
\newtheorem{corollary}[theorem]{Corollary}  %Corollaries are numbered with theorems
\newtheorem{lemma}[theorem]{Lemma}

\theoremstyle{definition}
\newtheorem*{definition}{Definition}      %Definitions are not numbered.

\newtheoremstyle{dotless}{}{}{}{}{\bfseries}{}{ }{}
\theoremstyle{dotless}
\newtheorem*{dLPO}{\sf LPO}     %Definitions associated with combinatorial properties
\newtheorem*{dLGk}{{\sf LG}$k$}
\newtheorem*{dGCk}{{\sf GC}$k$}
\newtheorem*{dTCk}{{\sf TGC}$k$}
\newtheorem*{dWKL}{{\sf WKL}}
\newtheorem*{dWKn}{{\sf WKL}$n$}
\newtheorem*{dLLPO}{\sf LLPO}

\newtheorem*{dWF}{$\wf$}
\newtheorem*{dS}{${\sf S}$}
\newtheorem*{dSL}{${\sf S}_L$}
\newtheorem*{dSLv}{$\slv$}
\newtheorem*{dRC}{$\sf{RC}$}
\newtheorem*{dD}{$\sf{D}$}

 \usepackage{enumitem}
\setlist[enumerate]{label=\rm{(\arabic*)}, ref=\arabic*}

\newcommand{\nat}{\mathbb N}

\newcommand{\rca}{{\sf RCA}_0}
\newcommand{\aca}{{\sf ACA}_0}
\newcommand{\atr}{{\sf ATR}_0}
\newcommand{\poo}{{\Pi^1_1 {\text -}{ \sf{CA}}_0}}
\newcommand{\rcaw}{{\sf RCA}_0^\omega}
\newcommand{\wkl}{{\sf WKL}}
\newcommand{\seq}[1]{{\langle #1\rangle}}
\newcommand{\haw}{{\widehat{\mathsf{E}\text{-}\mathsf{HA}}}\strut^\omega_{\raise2pt\hbox{\scriptsize${\mathord{\upharpoonright}}$}}}

\newcommand{\probP}{{\sf P}{\rm :}\forall x (p_1 (x) \to \exists y \, p_2 (x,y))}

\newcommand{\probQu}{{\sf Q}{\rm :}\forall u (q_1 (u) \to \exists v \, q_2 (u,v))}
\newcommand{\wlt}{\le_{W}}
\newcommand{\wslt}{<_{W}}
\newcommand{\swlt}{\le_{sW}}
\newcommand{\weq}{\equiv_{W}}
\newcommand{\sweq}{\equiv_{sW}}

\newcommand{\cat}{{^\smallfrown}}
\newcommand{\dotminus}{\frac{\cdot}{~~}}
\newcommand{\lgk}[1]{{\sf LG}{#1}}
\newcommand{\lpo}{{\sf{LPO}}}
\newcommand{\gck} [1] {{\sf GC}{#1}}
\newcommand{\tgc} [1] {{\sf TGC}{#1}}
\newcommand{\scd}{{\sf seq}}
\newcommand{\wkn}[1]{{\sf WKL}{#1}}
\newcommand{\llpo}{{\sf LLPO}}
\newcommand{\wf}{{\sf WF}}
\newcommand{\wfh}{\widehat{\sf{WF}}}
\newcommand{\pon}[1]{{\sf{P}}{#1}}

\newcommand{\slv}{{\sf S}_{\vec L}}

\begin{document}

\maketitle

\begin{abstract}
We extend a study by Lempp and Hirst of infinite versions
of some problems from finite complexity theory, using
an intuitionistic version of reverse mathematics and
techniques of Weihrauch analysis.
\end{abstract}

An early article of Hirst and Lempp \cite{hl} was motivated by the following example.
The problem of determining if a finite graph has a Hamiltonian path is NP complete,
while the problem for Eulerian paths is in the class P.  In reverse mathematics,
the problem of selecting which infinite graphs in an infinite sequence have Hamiltonian
paths is equivalent to $\poo$, while the corresponding problem for Eulerian graphs
is equivalent to $\aca$.  While very evocative, Hirst and Lempp showed that this
parallel does not generally hold.  In conjunction with an AMS special session in
honor of Lempp's birthday, we revisited these early results using the tools
of Weihrauch analysis in two ways.

In the next section we concentrate on formalized Weihrauch reductions for
graph theoretic problems.  This approach is limited to problems that can be
expressed with particularly simple formulas, having certain subformulas that
are $\exists$-free.  The process yields both reverse mathematical and Weihrauch
reducibility results from a single argument.  
More importantly, it facilitates the use of techniques from Weihrauch analysis,
like parallelization, in reverse mathematics proofs.
%However, our primary interest is that it 
%allows the application in reverse mathematics proofs of techniques common in
%Weihrauch analysis, like parallelization.
Additionally, the intuitionistic formal systems
used admit proof mining \cite{kbook}, so terms corresponding to the reduction functionals could
be extracted from formal proofs.  If we view the proofs as a form of verification,
formal Weihrauch reduction could offer a framework for finding verified extensions
of a trusted library of routines.

The final section addresses Weihrauch analysis of stronger results, with parallelizations
at the $\poo$ level.  The formulas describing the problems are too complicated for
formal analysis techniques, so we turn to traditional Weihrauch analysis.  The results
extend the currently small catalog of Weihrauch problems at this level, for example like those
in \cite{leaf}.  Weihrauch analysis at the level of $\poo$ appears to parallel
familiar reverse mathematics closely, in contrast to recent results related to
$\atr$ by Kihara, Marcone, and Pauly \cite{kmp}. 

\section{Formalized Weihrauch analysis}
The following development of formalized Weihrauch reducibility draws from
work of Hirst and Mummert \cite{hm}.  Like them, we work in $i \rcaw$, Kohlenbach's \cite{k2001}
higher order axiomatization for reverse mathematics, restricted to intuitionistic logic.
Informally, we may view $i\rcaw$ as an axiomatization of intuitionistic arithmetic that can
prove the existence of computable functions and computable functionals.  For more
details of the axiomatization, see \cite{hm}.

We formalize a Weihrauch problem with a formula
$\forall x ( p_1 (x) \to \exists y\, p_2 (x,y))$ where $p_1(x)$ indicates that $x$ is an accepted input
and $p_2 (x,y)$ indicates that $y$ is a solution of $x$.  Here $y$ may be a function, set, or number,
depending on the problem.  For total problems, that is, those problems accepting all sets as inputs,
we can use the more simple representation $\forall x \exists y \, p(x,y)$.

Some of the results connecting provability and formalized Weihrauch reducibility are restricted
to a family of formulas called $\Gamma_1$ by Troelstra~\cite{troelstra}.
A formula is {\em $\exists$-free} if it is built from atomic formulas using only universal quantification
and the connectives $\land$ and $\to$.
(Troelstra includes $\bot$ as an atomic formula, so $\neg P$ abbreviates $P \to \bot$.)
The collection $\Gamma_1$ consists
of those formulas defined inductively by the following:
\begin{list}{$\bullet$}{}
\item  All atomic formulas are elements of $\Gamma_1$.
\item  If $A$ and $B$ are in $\Gamma_1$, then so are $A \land B$, $A \lor B$,
$\forall x A$, and $\exists x A$.
\item  If $A$ is $\exists$-free and $B$ is in $\Gamma_1$ then $\exists x A \to B$ is in $\Gamma_1$,
where $\exists x$ denotes a block of zero or more existential quantifiers.
\end{list}

\begin{definition}
Suppose $\probP$ and $\probQu$ are problems.  The reduction prenex formula for ${\sf Q} \wlt {\sf P}$
is the formula:
\[
S^{\sf P}_{\sf Q}:
\forall u \exists x \forall y \exists v ( q_1(u) \to ( p_1 (x) \land (p_2 (x,y ) \to q_2 (u,v) ) ) )
\]
If the problems are total, we can write $\sf P$ as
$\forall x \exists y \, p(x,y)$, $\sf Q$ as $\forall u \exists v \, q(u,v)$ and $S^{\sf P}_{\sf Q}$ as
$\forall u \exists x \forall y \exists v ( p(x,y) \to q(u,v))$.
\end{definition}

The implication $S^{\sf P}_{\sf Q} \to ({\sf P} \to {\sf Q})$ is provable in intuitionistic predicate calculus,
though the proof of the converse requires classical logic.  If the constituent formulas
$p_1$, $p_2$, $q_1$, and $q_2$ are all $\exists$-free, then the matrix of $S^{\sf P}_{\sf Q}$
is also $\exists$-free and $S^{\sf P}_{\sf Q}$ is in $\Gamma _1$.  This also holds for total problems.

If $\sf P$ and $\sf Q$ are Weihrauch problems, we can formalize ${\sf Q} \wlt {\sf P}$ by asserting
the existence of Skolem functions for the existential quantifiers of $S^{\sf P}_{\sf Q}$.
Using $\sf P$ and $\sf Q$ as in the definition, the relation ${\sf Q} \wlt {\sf P}$ can be
translated into the language of $i\rcaw$ as:
\[
\exists \Phi \exists \Psi \forall u \forall y
(q_1(u) \to (p_1 (\Phi(u))\land(p_2 (\Phi(u),y)\to q_2 (u, \Psi(y)))))
\]
Here the functionals $\Phi$ and $\Psi$ are of type $1\to1$, which is our primary motivation
for working in higher order subsystems.

Using the preceding notation, we can present the following slightly modified version
of Theorem 1 of Hirst and Mummert \cite{hm}.

\begin{theorem}\label{hm1}
Suppose $\sf P$ and $\sf Q$ are problems and
the reduction prenex formula
$S^{\sf P}_{\sf Q}$  is in $\Gamma_1$.
Then $i\rcaw\vdash S^{\sf P}_{\sf Q}$ if and only if
$i\rcaw \vdash {\sf Q} \wlt {\sf P}$.
\end{theorem}

\begin{proof}
The result follows from the application of Kohlenbach's \cite{kbook} proof mining technology
to extract terms corresponding to the Skolem functions.
Only one modification to the proof of Theorem 1 of Hirst
and Mummert \cite{hm} is needed.  The hypothesis of their theorem specifies that
the matrix of $S^{\sf P}_{\sf Q}$ is in $\Gamma_1$, which by the definition of
$\Gamma_1$ is equivalent to $S^{\sf P}_{\sf Q} \in \Gamma_1$.
\end{proof}

The following theorem is the intuitionistic analog of a conservation result
of Kohlenbach \cite{k2001}.

\begin{theorem}\label{cons}
The system $i\rcaw$ is conservative over $i\rca$ for second order formulas.
\end{theorem}

\begin{proof}
Kohlenbach's conservation result for $\rcaw$ over $\rca$ appears
as Proposition 3.1 in his article on higher order reverse mathematics \cite{k2001}.
Kohlenbach's proof is based on the formalization of the extensional model
of the hereditarily continuous functionals (ECF) as presented in
Section 2.6.5 of Troelstra's book \cite{troelstra}.  The arguments of
Troelstra (e.g.~2.6.12 and 2.6.20 of \cite{troelstra}) are carried out
in intutionistic systems.  Similarly, Kohlenbach's proof holds for $i\rcaw$ and $i\rca$.
For other related discussion, see Theorem 2.7
and Theorem 2.8 of Hirst and Mummert \cite{hmunif}.
\end{proof}

\subsection{Local graph coloring}\label{section2}

The following encoding of graphs is useful in exploring graph coloring problems.
We can view a countable infinite graph as having $\nat$ as its vertex set.
Fix a primitive recursive bijective pairing function, $p$, mapping $\nat$ onto $\{ (a,b) \mid a< b\}$, the
subset of $\nat \times \nat$ consisting of increasing pairs.  Both $p$ and $p^{-1}$ can be
defined in such a way that $i\rcaw$ proves their existence and salient properties.
Using this pairing function, any function $e : \nat \to \nat$ can be viewed as a characteristic
function for the edge set of a graph $G$.  If $e(n) \neq 0$ then the edge $p(n)$ is in $G$,
and if $e(n) = 0$ then $p(n)$ is not in $G$.  We will conflate $G$ with the set of
codes for the edges of $G$ and write $(a,b)\in G$ as a shorthand for
$e(p^{-1} (a,b) ) \neq 0$.

For any graph $G$ and any $m$,
let $G_m$ denote the finite subgraph
with vertices $\{0, \dots , m\}$ and edge set
$\{(a,b) \mid a< b \le m \land (a,b)\in G \}$.
We say that $G_m$ has a $k$-coloring if there is a finite function
$f:m\to k$ such that for all $a <b\le m$, $(a,b) \in G$ implies $f(a) \neq f(b)$.
Informally, vertices connect by an edge must have distinct colors.
The existence of $k$-colorings for (initial) finite subgraphs can be
formulated as a problem.

\begin{dLGk}
(Local $k$-coloring for graphs):  Fix $k$.
For a graph $G$ (encoded by a characteristic function for its edge set), there
is a value $\lgk{k}(G)$ such that $\lgk{k}(G)=0$ implies that for every $m$ the subgraph
$G_m$ has a $k$-coloring, and $\lgk{k}(G)=m>0$ implies that
$G_m$ has no $k$-coloring and $G_{m-1}$ has a $k$-coloring.
\end{dLGk}

Let $c(G,m)$ denote a primitive recursive function such that
$c(G,m) = 1$ if $G_m$ has a $k$-coloring and $c(G,m)=0$ if $G_m$ has no
$k$-coloring.  Using this function, we can formalize the predicate $\lgk{k}(G)=n$
as the following $\exists$-free formula,
\[
\begin{split}
(n=0 \to \forall m (c(G,m) = 1) )\land & \\
(n>1 \to  (c(G,n)= & 0 \land \forall t (t<n \to c(G,t)=1)))
\end{split}
\]
and note that $\lgk{k}$ is a total problem of the form
$\forall G \exists n ( \lgk{k}(G) = n )$.
We will compare $\lgk{k}$ to a version of the limited principle of omniscience.

\begin{dLPO}
(Limited principle of omniscience):  For every $p: \nat \to \nat$ there is a value
$\lpo (p)$ such that $\lpo (p) = 0$ implies that $\forall k ~p(k)>0$, and
$\lpo (p) =m>0$ implies that $p(m-1) = 0$ and for all $t<(m-1)$, $p(t) > 0$.
\end{dLPO}

We can formalize the predicate $\lpo(p)=n$ as the following $\exists$-free formula,
\[
(n=0 \to \forall k ( p(k) > 0 ))\land (n>0 \to (p(n-1)=0 \land \forall t(t<(n-1)\to p(t)>0)))
\]
and note that $\lpo$ is a total problem of the form $\forall p \exists n (\lpo(p)=n )$.
This version of $\lpo$ differs from that presented in the survey by
Brattka, Gherardi, and Pauly \cite{bgp}.  It is Weihrauch equivalent to their version,
but not strongly Weihrauch equivalent, as the range of their version includes only $\{0,1\}$
and the range of this version includes all of $\nat$.  However, for our purposes it is desirable
to have the underlying predicate be $\exists$-free.

\begin{lemma}\label{LGL}
$( i\rcaw )$
For each $k \ge 1$, both $S^{\lgk{k}}_\lpo$ and $S^\lpo_{\lgk{k}}$ hold.
\end{lemma}

\begin{proof}
Because both $\lpo$ and $\lgk{k}$ are total, we can use the simple form of the
reduction prenex formula.  Thus $S^{\lgk{k}}_\lpo$ is
\[
\forall p \exists G \forall y \exists v (\lgk{k}(G) = y \to \lpo(p) = v).
\]
To prove this in $i\rcaw$, fix an instance $p$ of $\lpo$.  Define the graph $G$ as follows.
For each $m$, if $p(m) =0$, add the edges
$\{(s,t) \mid m\le s< t \le m+k\}$ to $G$.  These edges form a complete subgraph on
$k+1$ vertices, precluding any $k$-coloring of $G_{m+k}$.  If $p(m) \neq 0$, add no new edges
to $G$.  By this construction, if $\lgk{k}(G) = 0$, then for all $m$, $p(m) >0$, so $\lpo (p) = 0$.
On the other hand, if $\lgk{k}(G) = m >0$, then $m=n+k$ for some $n\ge 0$, and $p(n)=0$,
so $\lpo(p)=n+1=m-k+1$.

To show that $S^{\lpo}_{\lgk{k}}$, fix a graph $G$.  Define an instance $p$ of $\lpo$ as follows.
For each $m$, if the subgraph $G_m$ has a $k$-coloring, let $p(m) = 1$.  If $G_m$ has no $k$-coloring,
let $p(m) = 0$.  The truncated subtraction $\lpo (p)\dotminus 1$ yields a correct output for $\lgk{k}(G)$.
\end{proof}

By formalizing the transitivity of Weihrauch reduction we can extract additional results from the preceding lemma.

\begin{lemma}\label{TRANS}
$(i\rcaw )$  If $ {\sf P} \wlt {\sf Q}$ and ${\sf Q} \wlt {\sf R}$, then ${\sf P} \wlt {\sf R}$.
\end{lemma}

\begin{proof}
This is a formalization of the well-known property for Weihrauch reductions.  The
system $i\rcaw$ can prove that compositions of given functionals exist, so the usual
proof holds.
\end{proof}

\begin{theorem}\label{LGT}
$(i\rcaw )$  For any $j \ge k \ge 1$, $\lpo \weq \lgk{k} \weq \lgk{j}$.
\end{theorem}

\begin{proof}
By our formalization of $\lpo$ and $\lgk{k}$, the corresponding predicates are $\exists$-free.
Consequently, both $S^{\lgk{k}}_\lpo$ and $S^{\lpo}_{\lgk{k}}$ are in $\Gamma_1$.
Applying Theorem~\ref{hm1} to the implications of Lemma~\ref{LGL} yields proofs in
$i\rcaw$ that $\lpo \wlt \lgk{k}$, $\lgk{k} \wlt \lpo$, and so $\lpo \weq \lgk{k}$.
This holds for all $j$ and $k$, so by Lemma \ref{TRANS} we have $\lgk{k} \weq \lgk{j}$.
\end{proof}

One noteworthy consequence of the preceding theorem is the provability of the Weihrauch equivalence
of $\lgk{2}$ and $\lgk{3}$ in $i\rcaw$.  The finite combinatorial analogs of these problems are
2-colorability and 3-colorability of finite graphs, which are respectively polynomial time computable
and NP complete.  As was the case in the traditional reverse mathematics and computability
theoretic analysis of Hirst and Lempp \cite{hl}, formal Weihrauch
analysis of the related infinite problems does not distinguish between these problems.
The current setting does allow us to apply techniques of Weihrauch analysis, like the
application of Lemma~\ref{TRANS} in the preceding proof to obtain formal equivalences,
and then to translate these into proofs of implications in weak subsystems, as in the following
corollary.

\begin{corollary}\label{LGC}
$(i\rcaw)$  For every pair of problems $\sf P$ and $\sf Q$ from the list in Theorem \ref{LGT},
$S^{\sf P}_{\sf Q}$ holds.  Furthermore, the implication ${\sf P}\to {\sf Q}$ holds.
\end{corollary}

\begin{proof}
As noted above, the formulas $S^{\sf P}_{\sf Q}$ of Theorem~\ref{LGT}
are in $\Gamma_1$.  Applying the reverse implication Theorem~\ref{hm1} to the equivalences of Theorem~\ref{LGT}
proves  $S^{\sf P}_{\sf Q}$.  To justify the final sentence, note that these are total problems and can be written
as $\Pi^1_1$ formulas.  Over intuitionistic predicate calculus, $S^{\sf P}_{\sf Q}$ implies ${\sf P}\to {\sf Q}$.
\end{proof}

If $\sf P$ is a problem, the parallelization of $\sf P$, denoted by $\widehat{\sf P}$ is the problem
that accepts as input any infinite sequence of inputs for $\sf P$, and outputs the associated infinite
sequence of solutions for the input instances.  The following lemma is a formalization of a portion
of part 3
of Proposition 3.6 of Brattka, Gherardi, and Pauly \cite{bgp}.

\begin{lemma}\label{HAT}
$( i\rcaw )$  ${\sf P} \wlt {\sf Q}$ implies $\widehat{\sf{P}}\wlt \widehat{\sf{Q}}$.
\end{lemma}

\begin{proof}
Suppose that $\Phi$ and $\Psi$ are the functionals witnessing ${\sf P} \wlt {\sf Q}$.
If $\vec u = \langle u_i \rangle_{i \in \nat}$ is an input for $\widehat{\sf Q}$, define
$\hat \Phi$ by $\hat \Phi (\vec u) = \langle \Phi (u_i) \rangle_{i \in \nat}$.  Define
$\hat \Psi$ similarly.  Note that $i\rcaw$ suffices to prove the existence of $\hat\Phi$
and $\hat \Psi$, and also proves that they witness $\widehat{\sf{P}}\wlt \widehat{\sf{Q}}$.
\end{proof}

\begin{theorem}\label{GHW}
$( i\rcaw )$  For every $k \ge 1$, $\widehat\lpo \weq \widehat{\lgk{k}}$.
\end{theorem}

\begin{proof}
Apply Lemma \ref{HAT} to Theorem \ref{LGT}.
\end{proof}

If $\langle f_i \rangle_{i \in \nat}$ is a sequence of input functions for $\widehat\lpo$, we can
define a type 1 function on (codes of pairs of) natural numbers by $f(i,n) = f_i(n)$ and think of $f$ as a type 1 input for $\widehat \lpo$.  This modification
allows us to think of $\widehat \lpo$ as a $\Pi^1_2$ formula of second order arithmetic.
Similarly, we will conflate $\widehat {\lgk{k}}$ with its second order analog.

\begin{corollary}\label{GHRM}
$(i\rca )$ The following are equivalent:
\begin{enumerate}
\item $\aca$. \label{GHRM1}
\item $\widehat{\lpo}$\label{GHRM2}
\item $\widehat{\lgk{k}}$, where $k\ge 1$.\label{GHRM3}
\end{enumerate}
\end{corollary}

\begin{proof}
The equivalence of parts (\ref{GHRM2}) and (\ref{GHRM3}) follows from
the application of Theorem \ref{hm1} to the Weihrauch reductions included in Theorem \ref{GHW}, and the
fact that $S^{\sf P}_{\sf Q} \to ({\sf P} \to {\sf Q})$.  To show that
(\ref{GHRM1}) implies (\ref{GHRM2}), let $\langle f_i \rangle_{i \in \nat}$ be
an input for $\widehat{\lpo}$.  Arithmetical comprehension asserts the
existence of the set $\{ i \mid \forall n \, f_i(n) \neq 0 \}$ and the
corresponding characteristic function, which is a solution of the instance of $\widehat{\lpo}$.

To show that (\ref{GHRM2}) implies (\ref{GHRM1}), we adapt the familiar
application of Lemma III.1.3 of Simpson \cite{simpson}.  In $i\rca$ it suffices to
find the range of an arbitrary function that is injective, except for possibly repeatedly
taking the value $0$.  (This broader class of almost injective functions avoids a use of classical
logic in the proof of Lemma II.3.7, used by Simpson \cite{simpson} to prove the reversal.)
Let $f$ be such a function.
Define the sequence of functions $\langle f_i \rangle_{i \in \nat}$ by
$f_i(m) = 1$ if $\forall t< m (f(t) \neq i)$ and $f_i (m) = 0 $ otherwise.
The solution to the $\widehat{\lpo}$ problem $\langle f_i \rangle_{i \in \nat}$
is the characteristic function of the range of $f$.

Finally, the formulas of Corollary~\ref{GHRM} are all second order, so by Theorem~\ref{cons}
the equivalences are provable in the second order system $i\rca$.
\end{proof}

Summarizing, we used higher order formalized Weihrauch techniques in Lemma \ref{HAT}
and extracted traditional second order reverse mathematics results in Corollary \ref{GHRM}.

\subsection{Colorings as outputs}

We can reformulate the graph coloring problems so that the outputs are infinite graph colorings.
We will continue to use the notation of the previous subsection.  In particular, $G_m$ denotes the finite
subgraph of $G$ consisting of the first $m+1$ vertices and their edges in $G$.

\begin{dGCk}
($k$-coloring of graphs):
If $e:\nat \to \nat$ codes the edge set for a graph $G$ and every finite
subgraph $G_m$ has a $k$-coloring, then there is a function $f: \nat \to k$
that is a $k$-coloring of $G$.
\end{dGCk}

The predicate asserting that $G_m$ is $k$-colorable can be written using bounded quantifiers, with
$G$, $k$, and $m$ as parameters.  The system $i\rcaw$ proves the existence of the primitive recursive
characteristic functional for this predicate, so the assertion that every $G_m$ has a $k$-coloring can
be formalized with a $\exists$-free formula.  Note that $f$ is a $k$-coloring of $G$ if and only if
for each $i \neq j$, if $(i,j)$ is an edge in $G$ then $f(i)\neq f(j)$, which is also $\exists$-free.
Thus we can write an $\exists$-free formula ${\sf{GC}}k(G,f)$ asserting that $f$ is the solution
to the ${\sf{GC}}k$ problem for $G$.

In our study of problems involving infinite trees, we will use the the following encoding.
Fix a suitable bijective coding function $\scd$ mapping $\nat$ onto $2^{<\nat}$, the set
of finite sequences of zeros and ones.  Any function $t: \nat \to \nat$ can be viewed as
a code for a binary tree $T$ using the following convention.  The finite sequence $\scd (n)$ is
in $T$ if and only if $t(n)>0$ and for every initial segment $\sigma$ of $\scd (n)$,
$t(\scd^{-1} (\sigma) )>0$.  Informally, the function $t$ is the characteristic function for a
set that includes the sequences of $T$ and omits the immediate successors of any leaves of $T$.
Using this encoding, we have the following formalization of Weak K\"o{}nigs's Lemma as a total
problem.

\begin{dWKL}
(Weak K\"o{}nig's Lemma):
If $T$ is a binary tree and for every $m$ there is a sequence $\sigma \in T$ of length $m$,
then there is a function $p$ that codes an infinite path through $T$.
\end{dWKL}

In the previous formalization, $p$ can be viewed as
a code for a subtree $P$ of $T$ such that each node in $P$ has precisely one immediate successor.
Using our standard techniques for representing primitive recursive functionals, it is easy to
formalize the assertion that for every $m$ the tree $T$ contains a binary sequence of length
$m$ as an $\exists$-free formula.  Note that the bound on the labels is essential for this.
The claim that $p$ is an infinite path is a universal statement, so we can write an $\exists$-free
formula ${\sf{WKL}}(T,f)$ asserting that $f$ is a solution to the $\sf{WKL}$ problem for $T$.

For values of $n$ greater than $2$, we can use a bijective coding $\scd _n$ of sequences of numbers
less than $n$, and formulate a corresponding version of $\wkl$.

\begin{dWKn}
(Weak K\"o{}nig's Lemma for trees in $n^{<\nat}$):
If $T$ is a tree with nodes labeled with numbers less than $n$,
and for every $m$ there is a sequence $\sigma \in T$ of length $m$,
then there is a function $p$ that codes an infinite path through $T$. 
\end{dWKn}

Note that $\wkl$ and $\wkn{2}$ denote the same problem.  As with $\wkl$, we can formalize
$\wkn{2}$ in the language of $i\rcaw$ with an $\exists$-free formula.

\begin{lemma}\label{GCL0}
$(i\rca)$ For each $n\ge 2$,  $S^{\wkl}_{\wkl n}$ and $S^{\wkl n}_ {\wkl}$ hold.
\end{lemma}

\begin{proof}
For $n\ge 2$, every binary tree is an $n$-ary tree, so $S^{\wkl n}_ {\wkl}$ holds trivially.
To prove $S^{\wkl}_{\wkl n}$, let $T$ be an $n$-ary tree.  Any sequence of $m$ values less
than or equal to $n$ can be mapped to a binary sequence of length $m\cdot n$ in the manner of
the following example.  Consider $\langle 2, 0 , 3 \rangle$ as a sequence of numbers less than or equal to $3$.
Replace each value $k$ with a block of length $3$ consisting of $k$ ones padded on the right with zeros.
In our example $\langle 2, 0 , 3 \rangle$ becomes $\langle 1,1,0, \,\, 0,0,0,\,\, 1, 1, 1\rangle$.  This
process transforms any $n$-ary tree $T$ into a binary tree that is well-founded if and only if $T$ is.
Furthermore, summing successive length $n$ blocks of any infinite path through the binary tree yields a path
through $T$.
\end{proof}

We will use the following formalization
of $\llpo$, which is defined for all functions from $\nat$ into $\nat$.
 Na\"i{}vely, if a function $p$ has a positive value, then $\llpo$ returns
 the flip of the parity of that first positive value.
 
 \begin{dLLPO}(Lesser limited principle of omniscience):
For any $p:\nat \to \nat$
there is a value $\llpo (p)$ such that 
%if $\llpo (p) =0$ then if there is a $j$ such that $p(j) \neq 0$, then the least such $j$ is odd, and
%if $\llpo (p) \neq 0$, then $\llpo (p) = 1$ and if there is $j$ such that $p(j) \neq 0$, then the least such $j$ is even.
$\llpo (p)=0$ implies if $p(j) \neq 0$ and $\forall t< j (p(j)=0)$ then $j$ is odd, and
$\llpo (0) \neq 0$ implies $\llpo(p) =1$ and if $p(j) \neq 0$ and $\forall t< j (p(j)=0)$ then $j$ is even.
\end{dLLPO}

Our definition is chosen so that for functions $p$ with ranges that are non-zero
in at most one place, the value of $\llpo (p)$ matches the
definition of Brattka, Gherardi, and Pauly \cite{bgp}*{\S7.2}.
By mapping functions from $\nat$ to $\nat$ to the characteristic function for the
least element of their domain taking a positive value, routine verification shows
that our total version of $\llpo$ is strongly Weihrauch equivalent to their version.

The predicate ``$j$ is odd'' can be formalized by $\forall m (j \neq 2m)$.  Similarly,
``$j$ is even'' can be formalized by $\forall m (j \neq 2m+1 )$.  Consequently,
$\llpo (p) =n$ can be formalized by the $\exists$-free formula:
\begin{align*}
&( n=0  \to 
\forall j ([p(j)\neq 0 \land \forall t (t<j \to p(t)=0)] \to \forall m (j \neq 2m))) \land \\
&( n\neq 0 \to 
(n=1 \land \forall j( [p(j)\neq 0 \land \forall t (t<j \to p(t)=0)] \to \forall m (j \neq 2m+1)))
)
\end{align*}

The problem $\llpo$ can be parallelized, resulting in a problem that is
Weihrauch equivalent to $\wkl$, as noted in part of
Theorem 7.23 of
Brattka, Gherardi, and Pauly \cite{bgp}.
The following results lead to a formalization of a fragment of their theorem,
and its connection to graph colorings.

\begin{lemma}\label{BGP}
$(i\rca )$  $S^{\sf \widehat\llpo}_{\sf \wkl}$ holds.% and $S^{\sf \wkl}_ {\sf \widehat\llpo}$ hold.
\end{lemma}

\begin{proof}
To prove $S^{\widehat\llpo}_{\wkl}$, suppose $T$ is a binary tree.  For every $\sigma \in 2^{<\nat}$, define
the instance $p_\sigma$ of $\llpo$ as follows.  If $\sigma \notin T$, then for each $n$,
$p(n) = 0$.  For $\sigma \in T$, let $p_\sigma (0) = p_\sigma (1) = 0$.
For $n> 0$, define $p_\sigma (2n)$ and $p_\sigma (2n+1 )$ as follows.
If $\sigma \cat 0$ has an extension of length $n$ in $T$ and $\sigma \cat 1$ has
no extension of length $n$ in $T$, let $p_\sigma (2n) = 0$ and $p_\sigma (2n+1 ) = 1$.
If $\sigma \cat 1$ has an extension of length $n$ in $T$ and $\sigma \cat 0$ has no
extension of length $n$ in $T$, let $p_\sigma (2n) = 1$ and $p_\sigma (2n + 1) = 0$.  Otherwise,
let $p_\sigma (2n) = p_\sigma (2n+ 1) = 0$.

If $\sigma \cat 0$ extends to an infinite path, but $\sigma \cat 1$ does not, then there is a first
$n>0$ such that $\sigma \cat 1$ has no extensions of length $n$.  In this case,
$\llpo ( p_\sigma ) = 0$.  Similarly, if $\sigma \cat 1$ extends to an infinite path but $\sigma \cat 0$
does not, $\llpo (p_\sigma ) =1$.  Thus, if $T$ is infinite, the sequence constructed by setting
$\sigma (0) = \llpo (\langle\, \rangle )$ and $\sigma (n+1) = \sigma (n) \cat \llpo (\sigma (n))$ for
$n \ge 0$ is an infinite path through $T$.
%To prove $S^{\wkl}_{\widehat\llpo}$, suppose $\seq{p_i}_{i \in \nat}$ is a sequence of instances of $\llpo$.  Define
%for every $i<n$, $\sigma (i)$ is a correct $\llpo$ solution for some extension of $p_i$ restricted to $n$.
%Note that $\sigma(i)$ is a correct solution of $p_i$ restricted to $n$ if and only if it is a correct
%solution to one of the following extensions:
%\[
%p(m) = 
%\begin{cases}
%p_i(m)&\text{if~}m<n\\
%1&\text{if~}m=n\\
%0&\text{otherwise}
%\end{cases}
%\qquad
%p^\prime(m) = 
%\begin{cases}
%p_i(m)&\text{if~}m<n\\
%1&\text{if~}m=n+1\\
%0&\text{otherwise}
%\end{cases}
%\]
%Consequently, $T$ is computable from $\seq{p_i}_{i \in \nat}$.  By the definition of $T$, if $s$ is an infinite path
%through $T$ then for each $i$, $s(i)$ is a correct $\llpo$ solution to $p_i$.
\end{proof}

%\begin{lemma}\label{GCL1}
%$( i\rcaw )$
%For each $k \ge 0$, $S^{{\sf P}_{i+1}}_{{\sf P}_i}$ holds for each problem in the following list:
%\[
%\begin{matrix}
%{\sf P}_0:  \widehat\llpo&
%{\sf P}_1:   \gck{2}&
%{\sf P}_2: \gck{3}&
%{\sf P}_3: \wkn{3}
%\end{matrix}
%\]
%\end{lemma}

\begin{lemma}\label{GCL1}
$(i\rca )$  $S^{\gck{2}}_{\widehat \llpo}$.
\end{lemma}

\begin{proof}
Suppose $p_i$ is an instance of $\llpo$.  Construct a subgraph $G_i$ with
vertices $u_i$ and $\{ v_{i,0} , v_{i,1}, \dots\}$ as follows.  
For all $k$, include the edge $(v_{i,k}, v_{i,k+1})$.  If $j$ is the least natural number
such that $p_i(j) \neq 0$, add the edge $(u_i , v_{i,j})$.  For any 2-coloring of $G_i$,
if $u_i$ and $v_{i,0}$ differ in color, then $1$ is a correct output for $\llpo$.  If $u_i$
and $v_{i,0}$ agree in color, then $0$ is a correct output.  For a sequence of instances
of $\llpo$, let $G$ consist of the disjoint union of the graphs for each sequence.
Any two coloring of $G$ yields a solution to the instance of $\widehat \llpo$.
\end{proof}

\begin{lemma}\label{GCLm}
$(i \rca )$  For $n \ge 2$, $S^{\gck{n}}_{\gck{2}}$.
\end{lemma}

\begin{proof}  Suppose $G$ is a locally $2$-colorable graph.  For $n>2$, create a single
copy of the complete graph on $n-2$ vertices, and connect each of its vertices to every
vertex of $G$.  The new graph is locally $n$-colorable, and the restriction of any $n$-coloring
to the vertices of $G$ yields a $2$-coloring of $G$.
\end{proof}

\begin{lemma}\label{GCL2}
$(i\rca )$  For $n \ge 2$, $S^{\wkn{n}}_ {\gck{n}}$ holds.
\end{lemma}

\begin{proof}
For any locally $n$-colorable graph $G$, $i\rca$ can prove that the tree of
sequences corresponding to $n$-colorings of the finite subgraphs $G_m$ is an instance of $\wkn{n}$.
\end{proof}

We can concatenate the lemmas to yield our main theorem on Weihrauch equivalences related to graph colorings.

\begin{theorem}\label{GCT}
$(i\rcaw )$  For $n\ge 2$, we have:
\[\wkn{n} \weq \wkl \weq \widehat\llpo \weq \gck{n} \weq \gck{2}\]
\end{theorem}

\begin{proof}
Applying Theorem \ref{hm1} to Lemma \ref{GCL0} and Lemma \ref{BGP}, we have $\wkn{n} \wlt \wkl \wlt \widehat{\llpo}$.
Applying Theorem \ref{hm1} to Lemma \ref{GCL1}, Lemma \ref{GCLm}, and Lemma~\ref{GCL2} yields
$\widehat{\llpo} \wlt \gck{2} \wlt \gck{n}\wlt \wkl{n}$.  The theorem follows by transitivity as provided by Lemma \ref{TRANS}.
\end{proof}

Given the formal Weihrauch reductions, we can extract the reverse mathematics consequences.

\begin{corollary}\label{GCC}
$(i\rca )$  The formula $S^{\sf P}_{\sf Q}$ holds for every problem $\sf P$ and
problem $\sf Q$ appearing in Theorem \ref{GCT}.  Furthermore the implication ${\sf P} \to {\sf Q}$ holds.
\end{corollary}

\begin{proof}
Imitate the proof of Corollary \ref{LGC}.
\end{proof}

\begin{corollary}\label{GCC2}
$(i\rca )$  For $k \ge 2$, the following are equivalent:
\begin{enumerate}
\item  $\wkl$.
\item  $\widehat \llpo$.
\item  $\wkn{k}$
\item  $\gck{k}$.
\end{enumerate}
\end{corollary}

\begin{proof}
Immediate from Corollary \ref{GCC}.
\end{proof}

Parallelization of $\gck{k}$ results in a Weihrauch equivalent problem.  This arises from the well-known
idempotence of the parallelization operator, which can be verified in the formal setting.

\begin{lemma}\label{WKLH}
$( i\rcaw )$  The hat operator is idempotent, that is, for any problem $\sf P$, $\widehat{\sf P}\weq \widehat{\widehat{\sf P}}$.  Consequently, for $k\ge 2$, $\gck{k} \weq \widehat {\gck{k}}$.
\end{lemma}

\begin{proof}
In $i\rcaw$, we can use the bijective pairing function to match each pair $(i,j)$ with an integer code.  Given a
sequence of sequences of problems $\langle\langle p_{i,j} \mid j \in \nat \rangle i\in \nat \rangle$, we match each problem
with an element of the sequence $\langle p_{(i,j)} \mid (i,j) \in \nat\rangle$.  Thus $\widehat{\widehat{\sf P}} \wlt \widehat{\sf P}$.
The reverse reduction is trivial.

From Theorem \ref{GCT} we know that $\gck{k} \weq \widehat \llpo$.  By Lemma \ref{HAT},
$\widehat{ \gck{k}} \weq \widehat {\widehat \llpo}$.  From the preceding paragraph, $\widehat {\widehat \llpo} \weq \widehat \llpo$,
so the desired equivalence follows by transitivity.
\end{proof}

In combination with Theorem \ref{GCT} we see that $\widehat{\gck{2}} \weq \widehat{\gck{3}}$.  These problems correspond
to reverse mathematical statements about infinite sequences  of 2-colorings and 3-colorings reminiscent of those of
Hirst and Lempp \cite{hl}.  The corresponding reverse mathematical results can be extracted in the usual fashion,
yielding the following corollary.

\begin{corollary}\label{GCC4}
$(i\rca )$  For $k\ge 2$, $\wkl$ is equivalent to the assertion that
for any sequence $\seq{G_i}_{i \in \nat}$ of infinite locally $k$-colorable graphs, there
is a function $f: \nat \times \nat \to \nat$ such that for each $i$, the function $f(i,n)$ is
a $k$-coloring of $G_i$.
\end{corollary}

The $k$-coloring problem defined at the beginning of this subsection does not accept
inputs which are not locally $k$-colorable.  The following alternative definition extends
the possible inputs to all graphs.

\begin{dTCk}(Total $k$-coloring of graphs):
Given a graph $G$, there is a function $f: \nat \to \nat$
such that $f(0)=0$ implies $f$ is a $k$-coloring of $G$ and $f(0)>0$ implies $G_{f(0)}$ has no $k$-coloring.
\end{dTCk}

The predicate asserting that $f$ solves $\tgc{k}$ for $G$ can be formalized as a $\exists$-free formula
in the parameters $k$, $f$, and $G$.
Altering the problem affects the Weihrauch analysis.

\begin{theorem}\label{TGC1}
$( i\rcaw )$  For $k\ge 2$, $\lpo\wlt \tgc{k} \wlt \widehat \lpo$.
\end{theorem}

\begin{proof}
For the first reduction, fix $k$ and suppose $p$ is an instance of $\lpo$.  Let $G$ be the graph
which is completely disconnected except for a completely connected subgraph on the vertices
$\{m, m+1, \dots , m+k\}$ if $p(m)=0$ and $\forall t< m (p(m) \neq 0)$, if such an $m$ exists.
The characteristic function for the edges of $G$ is uniformly computable from $p$.  Suppose
$f$ is a solution of $\tgc{k}$ for $G$.  Then $f(0)=0$ if and only if $\forall k (p(k) \neq 0)$.  Also,
if there is a least $m$ such that $p(m)=0$, then $G_{m+k}$ is the first initial subgraph of $G$
with no $k$-coloring.  In this case, $f(0)=m+k$ and the solution to the $\lpo$ problem $p$ is
obtained by subtracting the fixed value $k$.

For the second reduction, suppose $G$ is a graph.  Let $p_0 (n) = 1$ if for all $t\le n$ the initial
subgraph $G_k$ has a $k$-coloring, and let $p_0 (n) = 0$ otherwise.  For $i>0$, define $p_i$ as follows.
Let $\sigma_1 , \sigma_2 , \dots$ be the finite sequences in $k^{<\nat}$
that start with $0$ and occur in the tree of initial segments of
$k$ colorings of $G$.  If there are only finitely many such sequences, pad the list with copies of the empty
sequence.  Let $p_i (m) = 1$ if $\sigma_i$ extends to a $k$-coloring of $G_m$, and let $p_i (m) = 0$
otherwise.  Given a solution to the $\widehat \lpo$ problem $\langle p_i \rangle_{i \in \nat}$, we can compute a
solution $f$ to $\tgc{k}$ for $G$ as follows.  Let $f(0) =\lpo (p_0 )$.  If $ f(0)>0$, let $f(n) = 0$ for all $n>0$.
Otherwise, use the values of $\lpo (p_n )$ for $n>0$ to enumerate the nodes in a path through the tree of partial colorings
and assign the values of $f$ to match the nodes in the path.
This enumeration is computable because if $\lpo (p_n) = 0$, then the tree of partial
$k$-colorings extending $\sigma_n$ is infinite, and there is a least $m$ such that there is a $j<k$ with
$\sigma_m=\sigma_n ^\frown j$ and $\lpo ( p_m ) = 0$.
\end{proof}

\begin{corollary}\label {TGC2}
$( i\rcaw )$  For $k\ge 2$, $\widehat {\tgc{k}} \weq \widehat \lpo$ and $\widehat  {\gck{k}} \weq \widehat \llpo$.
\end{corollary}

\begin{proof}
Apply parallelization and idempotence to the reductions in Corollary \ref{TGC2} and Corollary \ref{GCC2}.
\end{proof}

Because the associated predicates are $\exists$-free, the Weihrauch results of Corollary \ref{TGC2}
can be converted to a reverse mathematics result
after the fashion of Corollary \ref{GCC4}.

\begin{corollary}\label{TGC3}
$(i\rca )$  For $k\ge 2$, $\aca$ is equivalent to the assertion that
for any sequence $\seq{G_i}_{i \in \nat}$ of graphs, there
is a function $f: \nat \times \nat \to \nat$ such that for each $i$, either $f(i,0)=m>0$ and the $m^\text{th}$ initial subgraph
of $G_i$ is not $k$-colorable or $f(i,0)=0$ and the function $f(i,n)$ is
a $k$-coloring of $G_i$.
\end{corollary}

The previous corollary again demonstrates the use of formal Weihrauch techniques to derive reverse mathematical results.
We close this section by illustrating a limitation of the formal approach.  In traditional Weihrauch analysis, we
can show that ${\sf P} \not \wlt {\sf Q}$ by showing that there are no computable functionals witnessing the reduction.
The formal setting does not distinguish between computable witnesses and other functionals, so there
is no analogous argument.  Consequently, results about strict Weihrauch reducibility are not amenable to formalization in $i\rcaw$.

\begin{corollary}\label {TGC4}
For $k\ge 2$, $i\rcaw$ proves $\gck{k} \wlt \tgc{k}$.  In the non-formal setting, $\gck{k}\wslt \tgc{k}$.
\end{corollary}

\begin{proof}
$i\rcaw$ proves that $\gck{k} \wlt \tgc{k}$ using the identity functionals as witnesses.  To prove the
strict inequality of the second sentence, we may use classical logic and results from the
Weihrauch reducibility literature.
Suppose by way of contradiction that $\tgc{k} \wlt \gck{k}$.  Parallelization preserves the reduction.
By Corollary \ref{TGC2} and transitivity, $\widehat \lpo \wlt \widehat \llpo$, contradicting 
widely known results (e.g.~Theorems 7.23, 7.40, and 7.42 in \cite{bgp}).\end{proof}

\section{Weihrauch analysis related to $\poo$}

In this section, we apply more traditional techniques of Weihrauch analysis to 
problems from \cite{hl}.
The underlying formulas are not $\exists$-free, so the techniques
of the preceding section are not applicable.
The problems are all related to $\poo$.
The first is based on Lemma VI.1.1 of Simpson \cite{simpson}.

\begin{dWF}(Well founded trees):
Given a tree $T$ in $\nat^{<\nat}$ as input, output $0$ if $T$ contains an infinite path
and $1$ if it does not.
\end{dWF}
The next two problems are based on principles included in Theorem 3.4 of Hirst and Lempp \cite{hl}.

\begin{dS}(Isomorphic subgraph): %Q2 in earlier drafts of this section.
Given inputs of graphs $G$ and $H$, output $1$ if $H$ is isomorphic to a subgraph of $G$ and $0$ if it is not.
\end{dS}

Let $L$ denote the linear graph with vertex set $V = \{ v_i \mid i \in \nat\}$ and edges $E= \{(v_i , v_{i+1} )\mid i \in \nat \}$.

\begin{dSL}($L$ as a subgraph):
Given a graph $G$ as input, output $1$ if $L$ is isomorphic to a subgraph of $G$ and $0$ if it is not.
\end{dSL}

\begin{theorem}\label{PW1}
$\wf \sweq {\sf S}_L \sweq {\sf S}$.
\end{theorem}

\begin{proof}  First we will show that $\wf \swlt {\sf S}_L$.
Any tree $T$ in $\nat^{<\nat}$ can uniformly be converted to a corresponding graph, with the nodes of the tree as vertices and
edges between neighboring nodes.
Note that the tree $T$ contains an infinite path if and only if the linear graph $L$ is isomorphic
to a subgraph of the graph of $T$.  The output $1-{\sf S}_L (T)$ is equal to $\wf (T)$, so the post-processing does not
depend on $T$ and the reduction is strong.

The problem  ${\sf S}_L $ is a special case of $\sf S$, so it only remains
to show that ${\sf S} \swlt \wf$.  Suppose $G$ and $H$ are input graphs with
vertices $\{ g_i \mid i \in \nat \}$ and $\{h_i \mid i \in \nat\}$.
We can uniformly compute the tree $T$ of initial segments of isomorphisms between $H$ and subgraphs of $G$,
where nodes correspond to sequences $g_{i_0} , \dots , g_{i_n}$ such that the pairing of nodes $g_{i_j}$ and $h_j$
for $j\le n$ is an isomorphism of the induced subgraph of $H$ with a subgraph of $G$.  Any infinite path through $T$
is an isomorphism between $H$ and a subgraph of $G$.  Thus $1-\wf (T)$ is equal to ${\sf S}(G,H)$, yielding the
final strong Weihrauch reduction.
\end{proof}

Theorem 3.4 of \cite{hl} discusses sequences of graphs and trees.  Similar Weihrauch reductions follow by parallelization.

\begin{corollary}\label{PW2}
$\wfh \sweq \widehat {\sf S}_L \sweq \widehat {\sf S}$.
\end{corollary}

\begin{proof}
As noted in Proposition 3.6 part 3 of \cite {bgp}, parallelization is a strong closure operation with respect
to strong Weihrauch reducibility.  The parallelized equivalences follow immediately from the equivalences in
Theorem \ref{PW1}.
\end{proof}

The proof of Theorem 3.4 of \cite{hl} indicates that an instance of $\wfh$ can be related to isomorphic subgraph
problems for a single target graph $G$.  Let $L_n$ denote the linear graph $L$ with a cycle of size $n+3$ appended
to the first vertex.  For example, $L_0$ is a copy of $L$ with a triangle attached as a tag to the first vertex.

\begin{dSLv}(Tagged linear subgraphs of a fixed graph):
Given a graph $G$ as input, output a function $s: \nat \to 2$ such that $s(n) = 1$ if and only if $L_n$ is a subgraph of $G$.
\end{dSLv}

\begin{theorem}\label{PW3}
$\wfh \sweq \slv$.
\end{theorem}

\begin{proof}  Note that an input $G$ for $\slv$ corresponds exactly to the input $\seq{G, L_i} _{i \in \nat}$ for $\widehat {\sf S}$.
Thus $\slv \swlt {\widehat {\sf S}} \sweq \wfh$.  To show that $\wfh \swlt \slv$, we adapt the reversal from Theorem 3.4 of \cite{hl}.
Let $\seq{T_i}_{i \in \nat}$ be an input for $\wfh$.  For each $i$, let $G_i$ be the graph of $T_i$ with a cycle of size $i+3$ attached
to the root node.  The graph $G$ consisting of the disjoint union of the $G_i$ graphs is uniformly computable from
$\seq{T_i}_{i \in \nat}$.  Note that $L_i$ is isomorphic to a subgraph of $G$ if and only if $L_i$ is isomorphic to a subgraph
of $G_i$, which occurs exactly when $T_i$ is not well founded.  If $s:\nat \to 2$ is a solution to $\slv$ for this graph $G$, then
$w(n)=1-s(n)$ is a solution of $\wfh$ for $\seq{T_i}_{i \in \nat}$, completing the proof that $\wfh \swlt \slv$.
\end{proof}

Many other graphs could be substituted for the linear graph $L$ in the preceding discussion.  However, not all graphs
yield the same results.  For example, while ${\sf S}_L \sweq \wf$, finite graphs yield a weaker Weihrauch problem.

\begin{theorem}\label{PW4}
Suppose that $F$ is a finite graph with at least two vertices.  Then ${\sf S}_F \weq \lpo$.
\end{theorem}

\begin{proof}
We will prove the result for the version of $\lpo$ from section \S\ref{section2}.
For other versions of $\lpo$, the result can be strengthened to strong Weihrauch reduction.
Fix a finite graph $F$.  To show that ${\sf S}_F \wlt \lpo$, let $G$ be an input graph and
construct an input for $\lpo$ as follows.  For each $n$, if $F$ is not isomorphic to a subgraph of $G$ restricted
to the first $n$ vertices of $G$, set $p(n)=1$.  If $F$ is isomorphic to such a subgraph, set $p(n)=0$.  Thus
${\sf S}_F (G) = 1$ if and only if $\lpo (p) > 0$.

To prove that $\lpo \wlt {\sf S}_F$, we will consider two cases.  First suppose that $F$ is a finite graph with
$j$ vertices and at least one edge.  Let $p$ be an instance of $\lpo$.  Let $G$ be the graph with vertices
$\{v_i \mid i \in \nat \}$, and with edges corresponding to a copy of $F$ on vertices $v_m , \dots , v_{m+j-1}$
if $m$ is the least value such that $p(m)= 0$.  If no such $m$ exists, then $G$ is completely disconnected.
Note that each edge of $G$ depends only on an initial segment of the values of $p$, so $G$ is uniformly
computable from $p$.  If ${\sf S}_F (G) = 0$, output $0$ for $\lpo(p)$.
If ${\sf S}_F (G) = 1$, then the search for the least $m$ such that $p(m)=0$ will
succeed and the appropriate value of $\lpo (p)$ can be output.

The case when $F$ is completely disconnected is similar, using complementary graphs.  All possible edges
are added to $G$ until a zero of $p$ is discovered, at which point no additional edges are added.
\end{proof}

The fact that ${\sf S}_F \weq \lpo$ and ${\sf S}_L \weq \wf$ motivates the following question:  Is there a graph
$H$ such that $\lpo \wslt {\sf S}_H \wslt \wf$?  
Arno Pauly says yes, if we switch from isomorphism to subgraph embedding ({\sf SE}).
For the graph $H$ consisting of infinitely many copies of K3, $ {\sf SE}_H\equiv_{\sf W} \lpo^\prime$  \cite{pauly}.
Switching to embeddings does not change Theorem \ref{PW1}, for example.  Reed Solomon has
asked what other reductions differ in strength for various notions of embedding.

We can also consider the problem ${\sf S}^G$ for a fixed graph $G$, asking whether or not an input graph is isomorphic
to a subgraph of $G$.  For the complete graph $K$ and the totally disconnected graph $D$ it is easy to show that
${\sf S}^K$ and ${\sf S}^D$ are both Weihrauch equivalent to $\lpo$.  Theorem 3.2 of \cite {hl} shows that there is a
computable graph $G$ such that the set of indices of computable graphs that are isomorphic to a subgraph of $G$ is
$\Sigma^1_1$ complete.  This prompts us to ask:  Is there a graph $G$ such that ${\sf S}^G \sweq \wf$?
Indeed, is there even a graph $G$ such that $\lpo \wslt {\sf S}^G$?

Theorem 2.6 of \cite{hl} shows a connection between sequential versions of the following problems.
For a graph $G$ with vertex set $V$, we say $c:V\to \nat$ is a coloring of $G$ if whenever $(v_1 , v_2)$ is
an edge of $G$, $c(v_1 ) \neq c(v_2 )$.

\begin{dRC} (Repeated color):
Given a graph $G$ as an input, output $1$ if there is a coloring of $G$ that uses one color infinitely often and
output $0$ if there is no such coloring.
\end{dRC}

\begin{dD}(Disconnected subgraph):
Given a graph $G$ as an input, output $1$ if $G$ has an infinite completely disconnected subgraph and output
$0$ otherwise.
\end{dD}

\begin{theorem}\label{PW5}
${\sf RC} \sweq {\sf D} \sweq \wf$.
\end{theorem}

\begin{proof}
Any graph $G$ has a coloring as required by $\sf RC$ if and only if it has a subgraph as in $\sf D$.  Thus
${\sf RC} \sweq {\sf D}$.  Overloading notation and using $D$ to denote the totally disconnected graph,
we have ${\sf D} \sweq {\sf S}_D$.  The problem ${\sf S}_D$ is a special case of ${\sf S}$, so
${\sf D} \swlt {\sf S} \sweq \wf$.

To show that $\wf \swlt {\sf S}_D$, given a tree $T$ as input, define a graph $G$ as follows.  The vertices of
$G$ will consist of the nodes of $T$.  A pair $(v_1 , v_2 )$ is an edge in $G$ if and only if the corresponding
nodes of $T$ are incomparable.  (Nodes of $T$ are comparable if one extends the other in the tree.)  $G$ is
uniformly computable from $T$.  The completely disconnected graph $D$ is isomorphic to a subgraph of
$G$ if and only if $T$ contains an infinite collection of pairwise comparable nodes, that is, if and only if
$T$ is not well founded.  For $T$ and $G$ as described above, $\wf (T) = 1-{\sf S}_D (G)$, so
$\wf \swlt {\sf S}_D$, as desired.
\end{proof}

Appending a triangle to each tree node in the previous construction yields a proof that $\wf$ is Weihrauch
equivalent to ${\sf S}_G$ for the graph $G$ consisting of infinitely many disconnected triangles.
(This was proven independently by Arno Pauly \cite{pauly}.)  Thus
the graph from Pauly's embedding observation does not answer our question about isomorphic subgraphs.
Parallelizing the previous theorem yields Weihrauch equivalences mirroring the reverse mathematics of Theorem 2.6 of \cite{hl}.

\begin{corollary}
$\widehat{\sf RC} \sweq \widehat{\sf D} \sweq \wfh$.
\end{corollary}

\begin{proof}
Immediate from Theorem \ref{PW5} by Proposition 3.6 part 3 of \cite{bgp}.
\end{proof}

The application of Weihrauch analysis to the problems of this section yields no new
insights into the corresponding finite complexity theoretic relationships.  However, it does raise
some new Weihrauch analysis questions.  Analyzing infinite versions of other statements from
finite complexity theory might lead to additional interesting Weihrauch reductions.

\end{document}